\documentclass[11pt,a4paper]{article}
\usepackage{bm}
\usepackage{graphicx,xcolor}
\usepackage{hyperref}
\usepackage{ascmac}
\usepackage{amsmath}
\usepackage{amscd}
\usepackage{amssymb}
\usepackage{amsfonts}
\usepackage{amsthm}
\usepackage[all]{xy}
\usepackage{mathtools}
\usepackage{enumerate}
\usepackage{tikz}
\usepackage{hyperref}
\usepackage{thm-restate}
\usepackage[shortlabels]{enumitem}
\usepackage{tikz-3dplot}
\usepackage{appendix}
\usepackage{url}

\newtheorem{lem}{Lemma}[section]
\newtheorem{prop}{Proposition}[section]

\newtheorem{obs}{Observation}

\newtheorem{rem}{Remark}
\newtheorem{que}{Question}

\title{Simplex volumes in hyperplane arrangements}
\author{Koki Furukawa}

\begin{document}
\date{}
\maketitle

\begin{center}
{\footnotesize
Department of Applied Mathematics, \\
Faculty of Mathematics and Physics, Charles University, Czech Republic \\
\texttt{koki@kam.mff.cuni.cz}\\
}
\end{center}

\begin{abstract}
The distinct distances problem and the unit distance problem, posed by Erd\H{o}s \cite{erdos1946} in 1946, are classical problems in combinatorial geometry. 
They ask for the minimum number of distinct distances determined by 
$n$ points in $\mathbb{R}^d$ and the maximum number of pairs at unit-distance that such a set can determine, respectively.
These problems have been studied extensively for several decades, leading to numerous deep results.
We study the following dual setting:
\begin{itemize}
  \setlength{\parskip}{0cm} 
  \setlength{\itemsep}{-0.43cm} 
  \item Determine the maximum number of unit $d$-dimensional $d$-volume $d$-simplices formed by $n$ hyperplanes in $\mathbb{R}^d$.\\
  \item Determine the maximum number of minimum (or maximum) $d$-dimensional $d$-volume $d$-simplices formed by $n$ hyperplanes in $\mathbb{R}^d$.\\
  \item Determine the maximum number $D_d(n)$ such that any arrangement of $n$ hyperplanes in $\mathbb{R}^d$ in general position contains a subset of size $D_d (n)$ for which all induced $d$-dimensional simplices have distinct $d$-dimensional volumes. 

\end{itemize}

\end{abstract}

\section{Introduction}
For $d \geq 2$, a finite point set $X$ in $\mathbb{R}^d$ with $|X| \geq d+1$ is said to be in \textit{general position} if no $d+1$ points of $X$ lie on the same hyperplane, 
and we say that an arrangement of hyperplanes $\mathcal{A}$ in $\mathbb{R}^d$ with $|\mathcal{A}| \geq d+1$ is in \textit{general position} if every $d$ hyperplanes intersect at exactly one point and no $d+ 1$ hyperplanes have a common point. 

Throughout the paper, we use $O_d$ to denote the standard big-$O$ notation in which the constant may depend on $d$. We use analogous conventions for $\Omega_d$ and $\Theta_d$.
For $1\le k\le d$, we will refer to $k$-dimensional volume and $k$-dimensional simplex as \textit{$k$-volume} and \textit{$k$-simplex}.
Furthermore, by ``a $k$-simplex in a hyperplane arrangement in $\mathbb{R}^d$'', we mean a simplicial $k$-cell determined by $k+1$ hyperplanes of the arrangement.

\subsection{Erd\H{o}s's distinct distances problem}
In 1946, Erd\H{o}s~\cite{erdos1946} posed the question: what is the minimum number $g_d(n)$ of distinct distances determined by $n$ points in $\mathbb{R}^d$. This question is known as \emph{Erd\H{o}s's distinct distances problem}, and a variety of interesting results have been obtained. For the planar case $d=2$, Erd\H{o}s showed in the same paper an upper bound
$g_2(n) = O(n/\sqrt{\log n})$ by considering the $\sqrt{n} \times \sqrt{n}$ square lattice.
While there had been earlier progress on lower bounds (\cite{chung1984number, chung1992number, katz2004new, moser1952different, solymosi2001distinct, ekely1993crossing, tardos2003distinct}), a major breakthrough by Guth and Katz~\cite{guth2015erdHos}
proved $g_2(n)=\Omega(n/\log n)$.

In higher dimensions, a $d$-dimensional grid $n^{1/d} \times \cdots \times n^{1/d}$ gives $g_d(n) = O_d (n^{2/d})$, and this bound is conjectured to be optimal. The best known lower bound to date is due to Solymosi and Vu~\cite{solymosi2008near} which is $g_d(n)=\Omega_d (n^{2/d-2/(d(d+2))})$. 

A well-studied variant of this problem asks for the minimum number $g'_d(n)$ of distinct $d$-volumes of $d$-simplices spanned by $n$ points in $\mathbb{R}^d$, assuming that not all $n$ points lie on a common hyperplane.
 One easily sees that $g'_d(n)\le \lfloor (n-1)/d \rfloor$ by taking $d$ sets of about $n/d$ equally spaced points on $d$ parallel lines passing through the vertices of a fixed $(d-1)$-simplex.  Erd\H{o}s, Purdy, and Straus~\cite{erdos1982problem} conjectured that this construction is tight for all sufficiently large $n$.  Pinchasi~\cite{pinchasi2008minimum} proved this conjecture for the case $d = 2$, and Dumitrescu and Cs.T\'{o}th~\cite{dumitrescu2008number} later showed that $g'_d(n)=\Theta_d (n)$ for every $d$.

There is another related problem asking \textbf{(A)}: what is the maximum number $h_d(n)$ such that any set of $n$ points in $\mathbb{R}^d$ contains a subset $S$ of size $h_d(n)$ whose induced $d$-simplices all have distinct $d$-volumes.
The upper bound is again derived from the $n^{1/d} \times \cdots \times n^{1/d}$ grid, which gives $h_2(n) = O_d(n^{2/3})$ and $h_d(n) = O_d(n^{(d-1)/d})$ for $d \geq 3$.
Conlon et al.~\cite{conlon2015distinct} showed $h_d(n) = \Omega_d(n^{1/(2d+2)})$ for all $d \geq 2$.
 
\subsection{Erd\H{o}s's unit distance problem} 
The unit distance problem asks: what is the maximum number $u_d(n)$ of unit distances that a set of $n$ points in $\mathbb{R}^d$ can form.
For the case $d = 2$, Erd\H{o}s~\cite{erdos1946} obtained the bounds $n^{1 + c / \log \log n} < u_2(n) = O (n^{3/2})$ for some $c>0$ and conjectured that his lower bound is tight. 
Recently, this was disproved by OpenAI~\cite{openai2026planar}, which showed that there exists a constant $\varepsilon > 0$, such that $u_2 (n) \geq  n^{1+\varepsilon}$ for infinitely many $n$. Shortly after this result, Sawin~\cite{sawin2026explicit} made the explicit exponent, proving that $u_2 (n) \geq n^{1.014}$ for arbitrarily large $n$.
Nevertheless, the problem remains wide open. The best known upper bound $u_2(n) = O (n^{4/3})$ is due to Spencer, Szemerédi, and Trotter~\cite{spencer1984unit}.
Despite numerous efforts (\cite{clarkson1990combinatorial, pach2006forbidden, ekely1993crossing}), no further improvements have been found.
 In $3$-dimensions, the best known bounds are $u_3(n) = O(n^{3/2})$ \cite{kaplan2012unit, zahl2013improved} and $u_3(n) = \Omega(n^{4/3} \log \log n)$ \cite{erdos1960set}. For $d \ge 4$,  we have $u_d (n) = \Theta_d (n^2)$ (see \cite{erdHos1990variations, lenz1955zerlegung}).

This problem also has a well-studied direction of generalization.
In 1967, Oppenheim asked \textbf{(B)}: what is the maximum number of unit area triangles determined by $n$ points in the plane? See \cite{erdHos1996extremal}.
A simple lattice $\sqrt{\log n} \times (n/\sqrt{\log n})$ again gives a lower bound $\Omega (n^2 \log \log n)$ \cite{erdos1971some}.
Currently, the best known upper bound is $O(n^{2+2/9})$, due to Raz and Sharir \cite{raz2017number}. 
Dumitrescu and Cs.T\'{o}th \cite{dumitrescu2008number} extended this problem to three dimensions and obtained the following result: the number of unit volume tetrahedra ($3$-simplices) determined by $n$ points in $\mathbb{R}^3$ is $O(n^{7/2})$, and there is an $n$-point set achieving $\Omega (n^3 \log \log n)$.

Related problems \textbf{(C)}: are to determine the number of triangles of minimum/maximum area determined by a set of $n$ points studied by Braß, Rote, and Swanepoel \cite{brass2001triangles}. They proved that the maximum number of triangles of maximum area determined by $n$ points in the plane is exactly $n$, 
and the maximum number of triangles of minimum area determined by $n$ points in the plane is $\Theta(n^2)$.

Dumitrescu and Cs.T\'{o}th \cite{dumitrescu2008number} further generalized this problem and obtained the following results:
The number of tetrahedra of minimum volume spanned by $n$ points in $\mathbb{R}^3$ is
at most $2n^3/3 - O(n^2)$, and there is a configuration of $n$-points achieving $3n^3/16 - O(n^2)$. 
Furthermore, they proved that the maximum number of $d$-simplices of minimum volume spanned by $n$ points in $\mathbb{R}^d$ is $O_d(n^d)$ for all $d$.

\subsection{Dual setting}
In this paper, we investigate the following dual setting of \textbf{(A)}, \textbf{(B)} and \textbf{(C)}.
\begin{enumerate}
\item
The following is a dual setting of \textbf{(A)}:
\begin{que} \label{que1}
What is the maximum number $D_d(n)$ such that any arrangement of $n$ hyperplanes in $\mathbb{R}^d$ in general position contains a subset of size $D_d (n)$ whose induced $d$-simplices have all distinct $d$-volumes\text{?}
\end{que}
For the case $d=2$, Damásdi et al.~\cite{damasdi2020triangle} obtained a lower bound of $n^{1/5}$ under the stronger assumption that there are no six lines that are tangent to a common conic (note that five lines always have a 
common tangent conic).

\item The problem of generalizing \textbf{(B)} to higher dimensions is to determine the maximum number of unit $d$-volume $d$-simplices formed by $n$ points in $\mathbb{R}^d$. For this problem, we consider the following dual version.

\begin{que}
What is the maximum number $f_d (n)$ of unit $d$-volume $d$-simplices formed by $n$ hyperplanes in $\mathbb{R}^d$?
\end{que}
For the case $d = 2$, Dam\'{a}sdi et al.~\cite{damasdi2020triangle} proved that $f_2(n) = O(n^{9/4 + \varepsilon})$ for every fixed $\varepsilon > 0$, and $f_2(n) = \Omega(n^2)$. 
Their proof uses the incidence theorem between points and algebraic curves, proved by Sharir and Zahl~\cite{sharir2017cutting}.

\item We propose the following variants for the higher-dimensional generalization of \textbf{(C)} (that is, to determinie the maximum number of $d$-simplices of minimum/maximum $d$-volume determined by a set of $n$ points in $\mathbb{R}^d$). 
\begin{que} \label{que3}
What is the maximum number $m_d (n)$ of minimum $d$-volume $d$-simplices formed by $n$ hyperplanes in $\mathbb{R}^d$ \text{?}
\end{que}

For the case $d = 2$, Dam\'{a}sdi et al.~\cite{damasdi2020triangle} proved that 
$m_2 (n) = \Theta (n^2)$.

\begin{que} \label{que4}
What is the maximum number $M_d(n)$ of maximum $d$-volume $d$-simplices formed by $n$ hyperplanes in $\mathbb{R}^d$ \text{?}
\end{que}

As we mentioned above, with respect to the occurrence of the maximum area determined by $n$ points in the plane, it is known that the maximum number is exactly $n$. However, Damásdi et al.~\cite{damasdi2020triangle} showed that this phenomenon does not occur in the dual setting: $M_2 (n) > 7n/5 - O(1)$.
We note that they also showed an upper bound  $M_2 (n) \leq n(n-1)/3$.
\end{enumerate}

\noindent
\subsection{Our results}
We provide partial answers to Questions 1-4, especially for $d \geq 3$.
\begin{restatable}{thm}{UPfd} \label{UPfd}
For the maximum number $f_3 (n)$ of unit volume $3$-simplices (i.e. tetrahedra) formed by $n$ planes in $\mathbb{R}^3$, we have
$$
 f_3 (n) = O_\tau (n^{\frac{7}{2} + \tau})
$$
for every fixed $\tau > 0$.
\end{restatable}

\begin{restatable}{thm}{LBfd} \label{LBfd}
For the maximum number $f_d (n)$ of unit $d$-volume $d$-simplices formed by $n$ hyperplanes in $\mathbb{R}^d$, we have
$$
f_d (n) = \Omega_d (n^d).
$$
\end{restatable}

\begin{restatable}{thm}{md} \label{md}
For the maximum number $m_d (n)$ of minimum $d$-volume $d$-simplices formed by $n$ hyperplanes in $\mathbb{R}^d$, we have
$$m_d (n) = \Theta_d (n^d).$$
\end{restatable}

\begin{restatable}{thm}{Md} \label{Md}
For the maximum number $M_d (n)$ of maximum $d$-volume $d$-simplices formed by $n$ hyperplanes in $\mathbb{R}^d$, we have
$$
M_d (n) > \frac{7}{5} (n-d) - O(1).
$$
\end{restatable}

\begin{restatable}{thm}{Dd} \label{Dd}
For the largest subset of hyperplanes forming $d$-simplices of distinct $d$-volumes, we have
$$
D_2 (n) \ll n(\log n)^{-c}
$$ 
for some absolute constant $c > 0$,
For every $d \geq 3$, there is $c_d \in (0,1)$ such that
$$
D_d (n) \ll n e^{-(\log \log n)^{c_d}}.
$$
\end{restatable}

Theorem \ref{Dd} gives partial answers for Problem 5.4 and Problem 5.5 in \cite{damasdi2020triangle}.

The rest of the paper is organized as follows. We will prove Theorem \ref{UPfd}, \ref{LBfd} in the next section.
In Section \ref{section3} we will prove Theorem \ref{md}. Section \ref{section4} is devoted to proving Theorem \ref{Md}. 
Finally, in Section \ref{section5}, we will give a proof of Theorem \ref{Dd}.

\section{Unit volume $d$-simplices} \label{section2}

In this section, we consider the number of $d$-simplices with unit volume determined by an arrangement of $n$ hyperplanes, namely to investigate the magnitude of $f_d(n)$.
We have a trivial upper bound $O_d(n^{d+1})$ by selecting generic $d+1$ hyperplanes from the arrangement.
The following provides a reasonably better upper bound for $d=3$.
\UPfd*

\begin{proof}
Let $\mathcal{H}$ be a family of $n$ planes in $\mathbb{R}^3$. Fix two distinct non-parallel planes $A,B \in \mathcal{H}$. If no such $A$ and $B$ exist, then there is nothing to prove. We first estimate the number of unit volume tetrahedra using $A$ and $B$. Since an affine transformation preserves ratios of volumes, it suffices to estimate the number of tetrahedra of volume $1$ containing $A$ and $B$ under the assumption that
$$
A:x=0, B:y=0.
$$
Let $\mathcal{H}_{A,B} \subset \mathcal{H} \setminus \{A,B\}$ be the set of planes $H \in \mathcal{H} \setminus \{ A,B \}$ which are not parallel to the direction of the line $A\cap B$. For each $H \in \mathcal{H}_{A,B}$, write
$$
H : z = a_H x + b_H y + c_H.
$$
Furthermore, define the dual point of $H$ by $p_H := (a_H, b_H, c_H) \in \mathbb{R}^3$, and set
$$
\mathcal{P}_{A,B} := \{ p_H : H \in \mathcal{H}_{A,B} \}.
$$
Clearly, $|\mathcal{P}_{A,B}| \leq n-2$. For two distinct planes $C,D \in \mathcal{H}_{A,B}$, in order for $A,B,C,D$ to determine a non-degenerate tetrahedron (that is, a tetrahedron of nonzero volume), it is necessary that
$$
(a_D-a_C)(b_D-b_C)(c_D-c_C)\neq0.
$$
In this case, it is easy to see that the volume of the tetrahedron determined by $A,B,C,D$ is
$$
\frac{|c_D - c_C|^3}{6 | (a_D - a_C)(b_D - b_C)|}.
$$
Therefore, the condition that the tetrahedron determined by $A,B,C,D$ has volume $1$ is equivalent to the existence of $\varepsilon \in \{+1, -1 \}$ such that
$$
(c_D - c_C)^3 = 6 \varepsilon (a_D - a_C)(b_D - b_C)
$$
and
$$
(a_D-a_C)(b_D-b_C)(c_D-c_C)\neq0.
$$
For each $\varepsilon \in \{+1, -1 \}$ and each plane $H \in \mathcal{H}_{A,B}$, consider the surface in $\mathbb{R}^3$ defined by
$$
S_H^\varepsilon : (z - c_H)^3 = 6 \varepsilon (x - a_H)(y - b_H),
$$
and define the family of surfaces
$$
\mathcal{S}_{A,B}^\varepsilon := \{ S_H^\varepsilon : H \in \mathcal{H}_{A,B} \}.
$$
We now define a bipartite graph $G_{A,B}^{\varepsilon}$ on $\mathcal P_{A,B} \sqcup \mathcal S_{A,B}^{\varepsilon}$ by adding an edge $(p_K, S_H^\varepsilon) \in \mathcal P_{A,B} \times \mathcal S_{A,B}^{\varepsilon}$ if and only if
$
p_K \in S_H^\varepsilon
$
and
$$
(a_H-a_K)(b_H-b_K)(c_H-c_K)\neq0
$$
hold. (Here $H,K \in \mathcal{H}_{A,B}$.) Then each tetrahedron of volume $1$ containing $A$ and $B$ gives exactly one edge in each of $G_{A,B}^{+1}$ and $G_{A,B}^{-1}$. Conversely, for each $\varepsilon \in \{ +1, -1 \}$, every edge of $E(G_{A,B}^{\varepsilon})$ uniquely determines a tetrahedron of volume $1$ containing $A$ and $B$. Hence the number of tetrahedra of volume $1$ containing $A$ and $B$ is equal to $|E(G_{A,B}^{+1})| = |E(G_{A,B}^{-1})|$. We will use the following lemma.

\begin{lem} \label{lem:incidence}
Fix $\varepsilon \in \{ +1, -1 \}$. For any three distinct planes $C_1, C_2, C_3 \in \mathcal{H}_{A,B}$, there are at most $4$ points $p = (a,b,c) \in \mathcal{P}_{A,B}$ such that
$
p \in S_{C_i}^{\varepsilon}
$
and
$
(a-a_{C_i})(b-b_{C_i})(c-c_{C_i})\neq0
$
for every $i=1,2,3$.
\end{lem}

This lemma immediately implies that the bipartite graph $G_{A,B}^\varepsilon$ is $K_{5,3}$-free. 

Let $P, Q \subset \mathbb{R}^d$ be finite sets, and let $G = (P \sqcup Q,E)$ be a bipartite graph. We say that $G$ is a \textit{semi-algebraic bipartite graph of description complexity at most $t$} if there exist polynomials $f_1, \ldots, f_t \in \mathbb{R} [x_1, \ldots, x_{2d}]$ of degree at most $t$ and a Boolean function $\Phi$ in $t$ variables such that, for every $(p,q) \in P \times Q$, the condition $(p,q) \in E$ is equivalent to $\Phi (f_1 (p,q) \geq 0, \ldots, f_t (p,q) \geq 0) = 1$. The following proposition is a special case of a result of Fox et al.~\cite{fox2017semi}.

\begin{prop}[Fox et al.~\cite{fox2017semi}]
Let $P, Q \subset \mathbb{R}^d$ be finite sets, and let $G = (P \sqcup Q,E)$ be a semi-algebraic bipartite graph of description complexity at most $t$. If $G$ is $K_{k,k}$-free, then
$$
|E| = O_{d,t,k,\tau} \big((|P| |Q|)^{\frac{d}{d+1} + \tau} + |P| +  |Q| \big)
$$
for every sufficiently small constant $\tau>0$.
\end{prop}

Note that $\mathcal{S}_{A,B}^\varepsilon$ can be identified with the parameter set $\mathcal{Q}_{A,B} := \{ (a_H, b_H, c_H) : H \in \mathcal{H}_{A,B} \}$, which is a copy of $\mathcal{P}_{A,B}$. Hence, by definition, the description complexity of $G_{A,B}^{\varepsilon}$ is at most $4$. Moreover, $G_{A,B}^{\varepsilon}$ is in particular $K_{5,5}$-free. Applying the above proposition with $d=3$ and $k=5$, we obtain, 
$$
|E(G_{A,B}^\varepsilon)|
= O_\tau \big((|\mathcal{P}_{A,B}| |\mathcal{S}_{A,B}^\varepsilon|)^{\frac{3}{4} + \tau} + |\mathcal{P}_{A,B}| +  |\mathcal{S}_{A,B}^\varepsilon| \big)
= O_\tau (n^{\frac{3}{2} + 2 \tau})
$$
for every $\tau>0$.
Therefore, summing over all choices of $A,B \in \mathcal{H}$ and taking into account the overcounting of tetrahedra, we get
$$
f_3 (n) \leq \dfrac{\binom{n}{2}}{\binom{4}{2}} O_\tau (n^{\frac{3}{2} + 2 \tau}) = O_\tau (n^{\frac{7}{2} + 2 \tau}).
$$
\end{proof}

\begin{proof}[Proof of Lemma~\ref{lem:incidence}]
For the sake of simplicity, 
for each $i=1,2,3$, denote the dual point of the plane $C_i$ by $q_i=(a_i,b_i,c_i) \in \mathbb{R}^3$.
 By applying a translation, we may assume that $q_1=(0,0,0)$. Thus, with variables $a,b,c$, the three surfaces that we need to consider are
$$
S_1: c^3=6\varepsilon ab
$$
and
$$
S_i: (c-c_i)^3=6\varepsilon(a-a_i)(b-b_i),\quad i=2,3.
$$
It suffices to count the points $p = (a,b,c)$ satisfying $p \in S_i$ for all $i=1,2,3$, and
$$
abc \neq 0
$$
and
$$(a-a_i)(b-b_i)(c-c_i)\neq0, \quad i=2,3.$$
First, for each $i = 2,3$, every common point $(a,b,c)$ of $S_1$ and $S_i$ satisfies
$$
a b_i + a_i b = a_i b_i - \frac{(c- c_i)^3 - c^3}{6 \varepsilon}.
$$
Denote the right-hand side by $R_i (c)$. Then $R_i (c)$ is a polynomial in $c$ of degree at most $2$. Therefore, every common point $(a,b,c)$ of $S_1, S_2, S_3$ satisfies
$$
\begin{pmatrix} b_2 & a_2 \\ b_3 & a_3 \end{pmatrix} \begin{pmatrix} a \\ b \end{pmatrix} = \begin{pmatrix} R_2 (c) \\ R_3 (c) \end{pmatrix}.
$$

First suppose that $(a_2,b_2)$ and $(a_3,b_3)$ are linearly independent. Then there exist polynomials $U(c),V(c)$ of degree at most $2$ such that
$$
a=U(c),\qquad b=V(c).
$$
Substituting this into the equation of $S_1$, we obtain
$$
c^3-6\varepsilon U(c)V(c)=0.
$$
If this equation is not an identity in $c$, then
$$
c^3-6\varepsilon U(c)V(c)
$$
is a nonzero polynomial in $c$ of degree at most $4$, and hence there are at most $4$ possible values of $c$. For each $c$, the values of $a$ and $b$ are uniquely determined. Thus, there are at most $4$ points $(a,b,c)$ satisfying the required conditions. Next suppose that
$$
c^3-6\varepsilon U(c)V(c)=0
$$
is an identity. Then $U(c) V(c)$ is a cubic monomial in $c$, and hence there exist nonzero constants $\alpha,\beta$ such that either $U(c)=\alpha c,\ V(c)=\beta c^2$, or $U(c)=\alpha c^2,\ V(c)=\beta c$. In the former case, for each $i=2,3$,
$$
R_i (c) = b_i \alpha c + a_i \beta c^2 = a_i b_i + \frac{c_i}{2 \varepsilon} c^2 - \frac{c_i^2}{2 \varepsilon} c + \frac{c_i^3}{6 \varepsilon}.
$$
Moreover, since $c^3 - 6 \varepsilon \alpha \beta c^3 = 0$, comparing coefficients gives $(a_i, b_i, c_i) = (0,0,0)$, contradicting the fact that $C_i$ is distinct from $C_1$. The latter case similarly gives $(a_i, b_i, c_i) = (0,0,0)$.

Next suppose that $(a_2,b_2)=(a_3,b_3)=(0,0)$. Then, at every common point $(a,b,c)$ of $S_1$ and $S_i$, we have
$$
c^3=(c-c_i)^3.
$$
Thus $c_i=0$. Hence $q_i=(0,0,0)=q_1$, contradicting the fact that $C_i$ is distinct from $C_1$. Therefore, in this case there is no point satisfying the required conditions.

Finally suppose that $(a_2,b_2)$ and $(a_3,b_3)$ are linearly dependent, but $(a_2,b_2)=(a_3,b_3)=(0,0)$ does not hold. By relabeling if necessary, we may assume that $(a_2,b_2)\neq(0,0)$. Then there exists a constant $\gamma$ such that
$$(a_3,b_3)=\gamma(a_2,b_2).$$
In this case, the two linear equations have a simultaneous solution if and only if
$$
R_3(c)=\gamma R_2(c).
$$
If this equation is not an identity in $c$, then $R_3(c)-\gamma R_2(c)$ is a nonzero polynomial of degree at most $2$. Hence there are at most $2$ possible values of $c$. Moreover, for the points that we count, $abc\neq0$, so $c\neq0$. For such a fixed $c$, the pair $(a,b)$ is an intersection point of the line
$$
b_2a+a_2b=R_2(c)
$$
and the hyperbola
$$
ab=\frac{c^3}{6\varepsilon}.
$$
Since $c\neq0$, there are at most $2$ such intersection points. Therefore, in this case, there are at most $4$ points $(a,b,c)$ satisfying the required conditions.

It remains to consider the case where
$$
R_3(c)=\gamma R_2(c)
$$
is an identity. First, if $\gamma=0$, then $R_3(c)$ is identically $0$. In this case $(a_3,b_3)=(0,0)$, and $R_3(c)\equiv0$ implies $c_3=0$. Thus $q_3=(0,0,0)=q_1$, contradicting the fact that $C_3$ is distinct from $C_1$. Hence we may assume that $\gamma\neq0$. Comparing coefficients gives
$$
c_3=\gamma c_2,\qquad c_3^2=\gamma c_2^2.
$$
If $\gamma=1$, then $q_3=q_2$, contradicting the fact that $C_2$ and $C_3$ are distinct. Hence we may assume that $\gamma\neq1$. Then
$$
\gamma^2c_2^2=\gamma c_2^2,
$$
and therefore $c_2=0$. Thus also $c_3=0$. Comparing the constant terms in the identity $R_3(c)=\gamma R_2(c)$ gives
$$
a_3b_3=\gamma a_2b_2.
$$
On the other hand, since $(a_3, b_3) = \gamma (a_2, b_2)$, we have
$$
a_3 b_3=\gamma^2a_2 b_2.
$$
Hence $(\gamma^2 - \gamma) a_2 b_2 = 0$. Since $\gamma \neq 0, 1$, it follows that $a_2 b_2=0$.

If $a_2=0$, then $b_2\neq0$ as $(a_2,b_2)\neq(0,0)$. Since $c_2=0$, we have $R_2(c)=0$, and the equation
$$
a b_2+a_2 b=R_2(c)
$$
implies $a=0$. This contradicts $abc\neq0$. Similarly, if $b_2=0$, then $b=0$, which again contradicts $abc\neq0$.
\end{proof}

It is clear that a lower bound for $m_d(n)$ gives a lower bound for $f_d(n)$. See the lower bound construction in the next section.

\section{Minimum volume $d$-simplices} \label{section3}
In this section, we prove Theorem \ref{md}. Let us recall the statement.
\md*

Kobon Fujimura asked in his book (called ''The Tokyo puzzle'' \cite{Fujimura1978}) what is the largest number $K (n)$ of non-overlapping triangles whose sides lie on an arrangement of $n$ lines (Gr\"{u}nbaum \cite{grunbaum1972arrangements} had studied a similar problem slightly earlier than Fujimura, but only in the case of arrangements in the projective plane).
A first upper bound was established by Saburo Tamura, which is $K (n) \leq \lfloor \frac{n(n-2)}{3} \rfloor$.
In 2007, Cl\'{e}ment and Bader \cite{clement2007tighter} found a slightly improved upper bound:
$$
K (n) \leq 
\begin{cases*}
 \lfloor \frac{n(n-2)}{3} \rfloor - 1 & if $n \equiv 0,2 \pmod 6$, \\
 \lfloor \frac{n(n-2)}{3} \rfloor & otherwise.
\end{cases*}
$$
For even $n$, Bartholdi et al.~\cite{bartholdi2008simple} provided a slightly sharper upper bound of $\lfloor \frac{n}{3} (n -\frac{7}{3}) \rfloor$.
On the other hand, F\"{u}redi and Pal\'{a}sti \cite{furedi1984arrangements} having constructed an arrangement of $n$ lines which contains $\lfloor \frac{n(n-3)}{3} \rfloor$ non-overlapping triangles, which gives $K (n) \geq \lfloor \frac{n(n-3)}{3} \rfloor$.

We now consider a higher-dimensional analogue.
Let $K_d (n)$ be the largest number of simplicial $d$-cells in an arrangement of $n$ hyperplanes in $\mathbb{R}^d$. Note that $K (n) = K_2 (n)$. 

For $d=2$, it immediately follows that $m_2(n) \leq K_2(n)$ from the fact that cutting a triangle with a single line always forms a smaller triangle in its interior.
However, unfortunately, the same argument does not hold for $d \geq 3$. Indeed, a $d$-simplex cut by a hyperplane creates a new $d$-simplex with $d$-volume smaller than the original in the interior if and only if the hyperplane separates one vertex from all the others. See Figure \ref{cutting a simplex} for the case $d=3$.

\begin{figure}[htbp]
\tikzset{every picture/.style={line width=0.75pt}} 

\begin{tikzpicture}[x=0.75pt,y=0.75pt,yscale=-0.68,xscale=0.68]

\draw    (374,272.94) -- (590.79,325) ;
\draw    (493.43,85) -- (590.79,325) ;
\draw    (493.43,85) -- (374,272.94) ;
\draw    (493.43,85) -- (657,205.63) ;
\draw    (657,205.63) -- (590.79,325) ;
\draw  [dash pattern={on 0.84pt off 2.51pt}]  (374,272.94) -- (657,205.63) ;
\draw [color={rgb, 255:red, 0; green, 0; blue, 0 }  ,draw opacity=1 ][fill={rgb, 255:red, 155; green, 155; blue, 155 }  ,fill opacity=0.5 ]   (447.42,156.91) -- (538.94,118.79) -- (634.43,242.69) -- (502.48,305.07) -- cycle ;
\draw  [fill={rgb, 255:red, 0; green, 0; blue, 0 }  ,fill opacity=1 ] (445.69,156.91) .. controls (445.69,155.96) and (446.46,155.18) .. (447.42,155.18) .. controls (448.38,155.18) and (449.16,155.96) .. (449.16,156.91) .. controls (449.16,157.87) and (448.38,158.65) .. (447.42,158.65) .. controls (446.46,158.65) and (445.69,157.87) .. (445.69,156.91) -- cycle ;
\draw  [fill={rgb, 255:red, 0; green, 0; blue, 0 }  ,fill opacity=1 ] (500.74,305.07) .. controls (500.74,304.12) and (501.52,303.34) .. (502.48,303.34) .. controls (503.44,303.34) and (504.21,304.12) .. (504.21,305.07) .. controls (504.21,306.03) and (503.44,306.81) .. (502.48,306.81) .. controls (501.52,306.81) and (500.74,306.03) .. (500.74,305.07) -- cycle ;
\draw  [fill={rgb, 255:red, 0; green, 0; blue, 0 }  ,fill opacity=1 ] (537.2,118.79) .. controls (537.2,117.83) and (537.98,117.06) .. (538.94,117.06) .. controls (539.9,117.06) and (540.67,117.83) .. (540.67,118.79) .. controls (540.67,119.75) and (539.9,120.52) .. (538.94,120.52) .. controls (537.98,120.52) and (537.2,119.75) .. (537.2,118.79) -- cycle ;
\draw  [fill={rgb, 255:red, 0; green, 0; blue, 0 }  ,fill opacity=1 ] (632.69,242.69) .. controls (632.69,241.73) and (633.47,240.96) .. (634.43,240.96) .. controls (635.39,240.96) and (636.17,241.73) .. (636.17,242.69) .. controls (636.17,243.65) and (635.39,244.42) .. (634.43,244.42) .. controls (633.47,244.42) and (632.69,243.65) .. (632.69,242.69) -- cycle ;
\draw    (4,275.94) -- (220.79,328) ;
\draw    (123.43,88) -- (220.79,328) ;
\draw    (123.43,88) -- (4,275.94) ;
\draw    (123.43,88) -- (287,208.63) ;
\draw    (287,208.63) -- (220.79,328) ;
\draw  [dash pattern={on 0.84pt off 2.51pt}]  (4,275.94) -- (287,208.63) ;
\draw [color={rgb, 255:red, 0; green, 0; blue, 0 }  ,draw opacity=1 ][fill={rgb, 255:red, 155; green, 155; blue, 155 }  ,fill opacity=0.5 ]   (140,129.27) -- (264.43,245.69) -- (132.48,308.07) ;
\draw  [fill={rgb, 255:red, 0; green, 0; blue, 0 }  ,fill opacity=1 ] (130.74,308.07) .. controls (130.74,307.12) and (131.52,306.34) .. (132.48,306.34) .. controls (133.44,306.34) and (134.21,307.12) .. (134.21,308.07) .. controls (134.21,309.03) and (133.44,309.81) .. (132.48,309.81) .. controls (131.52,309.81) and (130.74,309.03) .. (130.74,308.07) -- cycle ;
\draw  [fill={rgb, 255:red, 0; green, 0; blue, 0 }  ,fill opacity=1 ] (138.26,129.27) .. controls (138.26,128.31) and (139.04,127.53) .. (140,127.53) .. controls (140.96,127.53) and (141.74,128.31) .. (141.74,129.27) .. controls (141.74,130.22) and (140.96,131) .. (140,131) .. controls (139.04,131) and (138.26,130.22) .. (138.26,129.27) -- cycle ;
\draw  [fill={rgb, 255:red, 0; green, 0; blue, 0 }  ,fill opacity=1 ] (262.69,245.69) .. controls (262.69,244.73) and (263.47,243.96) .. (264.43,243.96) .. controls (265.39,243.96) and (266.17,244.73) .. (266.17,245.69) .. controls (266.17,246.65) and (265.39,247.42) .. (264.43,247.42) .. controls (263.47,247.42) and (262.69,246.65) .. (262.69,245.69) -- cycle ;
\draw    (140,131) -- (132.48,309.81) ;
\draw  [draw opacity=0][fill={rgb, 255:red, 74; green, 144; blue, 226 }  ,fill opacity=0.3 ] (140,129.27) -- (266.17,245.69) -- (220.79,328) -- (132.48,306.34) -- cycle ;

\end{tikzpicture}
\caption{Left: a cut creating a $3$-simplex in the interior of a $3$-simplex. Right: a cut not creating a $3$-simplex in the interior.}
\label{cutting a simplex}
\end{figure}
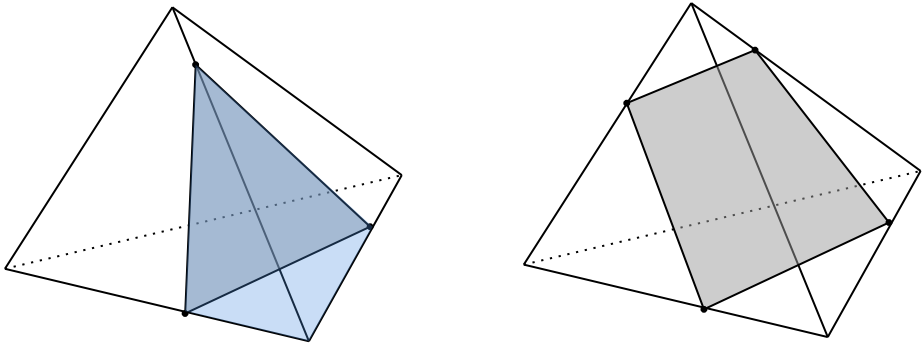

Nevertheless, finding bounds for $K_d(n)$ might be an interesting problem for its own sake.
\begin{prop} \label{Kobon}
For any $d \geq 3$, and $n \geq d+1$,
$$
K_d (n) \leq \frac{2^{d-1}}{(d+1)!} n^d + O_d (n^{d-1}).
$$
\end{prop}
The proof is presented in the Appendix.
We provide an upper bound for $m_d(n)$ by a different approach. 

\begin{prop} \label{tangent plane}
Let $H_1,\dots,H_d$ be hyperplanes in $\mathbb{R}^d$ meeting at a unique point and fix $V_0 > 0$.
Then, every hyperplane that form a $d$-simplex with $H_1, \ldots, H_d$ of $d$-volume $V_0$ are all tangent to one of two fixed algebraic hypersurfaces of degree $d$ which have $H_1,\ldots,H_d$ as asymptotes and each has exactly $2^{d-1}$ connected components.
Moreover, the $2^d$ connected components of these two hypersurfaces are contained one each in the $2^d$ connected components determined by $H_1,\ldots,H_d$.
\end{prop}

\begin{proof}
Note that an (invertible) affine transformation preserves hyperplanes, hypersurfaces of degree $d$ and ratios of volumes.
Fix a coordinate $(x_1, \ldots, x_d)$ and 
we may assume
$
H_i = \{x_i=0\} \quad (i=1,\ldots,d).
$
Fix $c\neq 0$ and consider the hypersurface
$$
S_c:=\Bigl\{x\in\mathbb{R}^d:\ \prod_{i=1}^d x_i=c\Bigr\}.
$$
Take an arbitrary point $p=(p_1,\dots,p_d)\in S_c$. Then $\prod_{i=1}^d p_i=c$, and $p_i\neq 0$ for all $i$.
The tangent hyperplane of $S_c$ at $p$ is given by
$
\sum_{i=1}^d \frac{x_i}{p_i}=d.
$
This tangent hyperplane together with $H_1,\dots, H_d$ forms a $d$-simplex $\Delta_p$ with vertices
$$
O = (0,\dots,0),\quad (dp_1,0,\dots,0),\quad (0,dp_2,0,\dots,0),\ \dots,\ (0,\dots,0,dp_d).
$$
Hence the volume of $\Delta_p$ equals
$$
\frac{1}{d!}\prod_{i=1}^d |dp_i|
=\frac{d^d}{d!}\,\Bigl|\prod_{i=1}^d p_i\Bigr|
=\frac{d^d}{d!}\,|c|.
$$
Since the volume of $\Delta_p$ depends only on $d$ and $|c|$, for any point of $S_{-c}$,
the tangent hyperplane at that point together with $H_1,\ldots,H_d$ also cuts out a $d$-simplex of volume
$\frac{d^d}{d!}\,|c|$.
The hypersurfaces $S_c$ and $S_{-c}$ have $2^{d-1}$ connected components each, and every connected components has
$H_1,\ldots,H_d$ as asymptotes.
Moreover, each of the $2^d$ regions determined by $H_1,\ldots,H_d$ contains exactly one connected component
of $S_c\cup S_{-c}$.
Consequently, any hyperplane that, together with $H_1,\ldots,H_d$, forms a $d$-simplex of volume
$\frac{d^d}{d!}\,|c|$ is tangent to exactly one connected component of $S_c\cup S_{-c}$.
One easily sees that the $2^d$ connected components of $S_c\cup S_{-c}$ are contained one each in the $2^d$ connected components determined by $H_1,\ldots,H_d$.
\end{proof}

\begin{proof}[Proof of Theorem \ref{md}]

\textit{Upper bound}: 
Let $\mathcal{H}=\{H_1,\dots,H_n\}$ be an arrangement of hyperplanes in $\mathbb{R}^d$, and let $m_0$ be the minimum $d$-volume of a $d$-simplex determined by $\mathcal{H}$.
Assume that there exists a $d$-tuple $(H_{i_1},\dots,H_{i_d})$ of hyperplanes in $\mathcal{H}$ meeting at exactly one point.
Applying Proposition~\ref{tangent plane} to $(H_{i_1},\dots,H_{i_d})$ with $V_0=m_0$, any hyperplane $T$ such that $\{H_{i_1},\dots,H_{i_d},T\}$ forms a $d$-simplex of $d$-volume $m_0$ must be tangent to one of the two fixed algebraic hypersurfaces of degree $d$.
These two hypersurfaces have $2^d$ connected components in total, and each of the $2^d$ connected components determined by $H_{i_1},\dots,H_{i_d}$ contains exactly one of them.

Let $C_R$ be the unique connected component contained in a fixed connected component $R$ in $\{H_{i_1},\dots,H_{i_d}\}$ .
It suffices to show that at most one hyperplane in $\mathcal{H}\setminus\{H_{i_1},\dots,H_{i_d}\}$ is tangent to $C_R$.
For a contradiction, we assume that there exist two distinct hyperplanes $T,T' \subset \mathcal{H}\setminus\{H_{i_1},\dots,H_{i_d}\}$ tangent to $C_R$.
Let $\tau$ be the $d$-simplex formed by $H_{i_1},\dots,H_{i_d},T$, and let $\tau'$ be the $d$-simplex formed by $H_{i_1},\dots,H_{i_d},T'$.
Note that $\tau$ and $\tau'$ are $d$-simplices of $d$-volume $m_0$.

Now consider the arrangement 
$\mathcal{A}=\{H_{i_1},\dots,H_{i_d},T,T'\}$.
A well-known result of Shannon~\cite{shannon1979simplicial} states that an arrangement of $N (\geq d+1)$ hyperplanes in general position in $\mathbb{R}^d$ has $N-d$ simplicial $d$-cells, and hence $\mathcal{A}$ has two simplicial $d$-cells.
Observe that these two simplicial $d$-cells have $T$ as a facet, and they do not lie on the same side of $T$.
Therefore, at least one of them is strictly contained in $\tau$ or $\tau'$.
This contradicts the minimality of $m_0$.
Hence, we have $m_d (n) \leq 2^d \binom{n}{d} = O_d (n^d)$.

\medskip
\noindent \textit{Lower bound}: 
Denote by $S_d$ the set of all permutations of $\{ 1, \ldots, d \}$ and let $\pi \in S_d$ be a permutation.
The Coxeter–Freudenthal–Kuhn (CFK) triangulation of the unit hypercube $[0,1]^d$ is a triangulation of $[0,1]^d$ into $d$-simplices $\sigma_{\pi} = [v_0 (\pi), \ldots, v_d (\pi)]$, whose vertices are given by
\[
v_0 (\pi) = (0, \ldots, 0), \quad v_i (\pi) = v_{i-1} (\pi) + e_{\pi (i)} \quad \text{for } i = 1, \ldots, d,
\]
where $e_1, \ldots, e_d$ is a standard basis of $\mathbb{R}^d$.
Note that the points $v_0 (\pi), \ldots, v_d (\pi) \in \mathbb{R}^d$ are vertices of $[0,1]^d$ and the CFK triangulation decomposes $[0,1]^d$ into $d!$ $d$-simplices with $d$-volume $1/d!$.
It is easy to observe that this decomposition is realized by $\binom{d}{2}$ hyperplanes
$
\{x_p - x_q = 0 : 1 \leq p < q \leq d \}.
$
See Figure \ref{CFK} and refer to \cite{coxeter1934discrete}, \cite{freudenthal1942simplizialzerlegungen}, \cite{kuhn1968simplicial} for further details

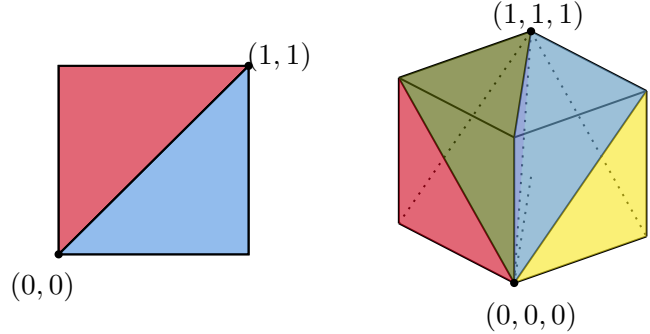
\begin{figure}[htbp]
\centering
\tikzset{every picture/.style={line width=0.75pt}} 

\begin{tikzpicture}[x=0.75pt,y=0.75pt,yscale=-0.6,xscale=0.6]

\draw    (371,130) -- (371,252) ;
\draw    (467,180) -- (467,302) ;
\draw    (371,130) -- (467,180) ;
\draw    (371,252) -- (467,302) ;
\draw    (481,91) -- (371,130) ;
\draw    (577,141) -- (577,263) ;
\draw    (481,91) -- (577,141) ;
\draw    (577,141) -- (467,180) ;
\draw    (577,263) -- (467,302) ;
\draw  [dash pattern={on 0.84pt off 2.51pt}]  (481,91) -- (577,263) ;
\draw    (577,141) -- (467,302) ;
\draw    (467,302) -- (371,130) ;
\draw  [dash pattern={on 0.84pt off 2.51pt}]  (481,91) -- (371,252) ;
\draw    (481,91) -- (467,180) ;
\draw  [dash pattern={on 0.84pt off 2.51pt}]  (481,213) -- (467,302) ;
\draw  [dash pattern={on 0.84pt off 2.51pt}]  (481,91) -- (467,302) ;
\draw  [draw opacity=0][fill={rgb, 255:red, 74; green, 144; blue, 226 }  ,fill opacity=0.6 ] (481,91) -- (577,141) -- (467,302) -- (467,180) -- cycle ;
\draw  [draw opacity=0][fill={rgb, 255:red, 248; green, 231; blue, 28 }  ,fill opacity=0.6 ] (577,263) -- (467,302) -- (577,141) -- (577,141) -- cycle ;
\draw  [draw opacity=0][fill={rgb, 255:red, 248; green, 231; blue, 28 }  ,fill opacity=0.1 ] (577,141) -- (467,302) -- (481,91) -- cycle ;
\draw  [draw opacity=0][fill={rgb, 255:red, 65; green, 117; blue, 5 }  ,fill opacity=0.6 ] (371,130) -- (481,91) -- (467,180) -- (467,302) -- cycle ;
\draw  [draw opacity=0][fill={rgb, 255:red, 208; green, 2; blue, 27 }  ,fill opacity=0.6 ] (371,130) -- (371,252) -- (467,302) -- cycle ;
\draw  [draw opacity=0][fill={rgb, 255:red, 208; green, 2; blue, 27 }  ,fill opacity=0.1 ] (481,91) -- (371,130) -- (467,302) -- cycle ;
\draw   (88,120) -- (246,120) -- (246,278) -- (88,278) -- cycle ;
\draw    (88,278) -- (246,120) ;
\draw  [fill={rgb, 255:red, 208; green, 2; blue, 27 }  ,fill opacity=0.6 ] (246,120) -- (88,278) -- (88,120) -- cycle ;
\draw  [fill={rgb, 255:red, 74; green, 144; blue, 226 }  ,fill opacity=0.6 ] (246,278) -- (246,120) -- (88,278) -- cycle ;
\draw  [fill={rgb, 255:red, 0; green, 0; blue, 0 }  ,fill opacity=1 ] (85.4,278.02) .. controls (85.39,276.59) and (86.54,275.41) .. (87.98,275.4) .. controls (89.41,275.39) and (90.59,276.54) .. (90.6,277.98) .. controls (90.61,279.41) and (89.46,280.59) .. (88.02,280.6) .. controls (86.59,280.61) and (85.41,279.46) .. (85.4,278.02) -- cycle ;
\draw  [fill={rgb, 255:red, 0; green, 0; blue, 0 }  ,fill opacity=1 ] (243.4,120.02) .. controls (243.39,118.59) and (244.54,117.41) .. (245.98,117.4) .. controls (247.41,117.39) and (248.59,118.54) .. (248.6,119.98) .. controls (248.61,121.41) and (247.46,122.59) .. (246.02,122.6) .. controls (244.59,122.61) and (243.41,121.46) .. (243.4,120.02) -- cycle ;
\draw  [fill={rgb, 255:red, 0; green, 0; blue, 0 }  ,fill opacity=1 ] (464.4,302.02) .. controls (464.39,300.59) and (465.54,299.41) .. (466.98,299.4) .. controls (468.41,299.39) and (469.59,300.54) .. (469.6,301.98) .. controls (469.61,303.41) and (468.46,304.59) .. (467.02,304.6) .. controls (465.59,304.61) and (464.41,303.46) .. (464.4,302.02) -- cycle ;
\draw  [fill={rgb, 255:red, 0; green, 0; blue, 0 }  ,fill opacity=1 ] (478.4,91.02) .. controls (478.39,89.59) and (479.54,88.41) .. (480.98,88.4) .. controls (482.41,88.39) and (483.59,89.54) .. (483.6,90.98) .. controls (483.61,92.41) and (482.46,93.59) .. (481.02,93.6) .. controls (479.59,93.61) and (478.41,92.46) .. (478.4,91.02) -- cycle ;

\draw (446,63.4) node [anchor=north west][inner sep=0.75pt]    {$( 1,1,1)$};
\draw (440,315.4) node [anchor=north west][inner sep=0.75pt]    {$( 0,0,0)$};
\draw (45,290.4) node [anchor=north west][inner sep=0.75pt]    {$( 0,0)$};
\draw (242,97.4) node [anchor=north west][inner sep=0.75pt]    {$( 1,1)$};

\end{tikzpicture}
\caption{The CFK triangulation of $[0,1]^2$ and $[0,1]^3$.}
\label{CFK}
\end{figure}

Now, we consider $d(n+1)$ hyperplanes given by
$
x_i = k \quad (i \in \{1, \ldots, d\}, \,\, k \in \{0, \ldots, n\}),
$
which form an arrangement that decomposes the hypercube $[0,n]^d$ into $n^d$ unit hypercubes. Furthermore, by adding $\binom{d}{2} \cdot (2n-1)$ hyperplanes
\[
\{x_p - x_q = t : 1 \leq p < q \leq d, \, t \in \mathbb{Z}, \, |t| \leq n-1 \},
\]
We can see that each unit hypercube admits the CFK triangulation (this is called K1 triangulation~\cite{REIA2025104237}).
The size of the resulting hyperplane arrangement is
$$
d(n+1) + \binom{d}{2} (2n-1) = d^2 n + \frac{d(3-d)}{2},
$$
and the number of (minimum volume) $d$-simplices is
$
d! \cdot n^d,
$
which gives us a lower bound $\Omega_d (n^d)$.

\end{proof}

\section{Maximum volume tetrahedra} \label{section4}

In this section, we show a lower bound for $M_d (n)$. A linear lower bound $M_d(n) \geq n-d$ can be easily obtained as follows: 
Fix coordinates $(x_1,\dots,x_d)$ in $\mathbb{R}^d$ and let
$H_i:=\{x_i=0\}$, $i=1,\dots,d$ and consider the branch
\[
\sigma:=\{x\in\mathbb{R}^d:\ \prod_{i=1}^d x_i =1,\ x_1,\ldots, x_d > 0\}.
\]
Choose $n-d$ distinct hyperplanes $T_1,\dots,T_{n-d}$ tangent to $\sigma$, with their tangent points contained in a sufficiently small neighborhood of some fixed point $u \in \sigma$.
For each $j$, the hyperplanes $H_1,\dots,H_d,T_j$ bound a $d$-simplex in the orthant $\{ (x_1, \ldots, x_d) : x_1, \ldots, x_d \geq 0\}$, and its $d$-volume equals $d^d/d!$.
Since the points of tangency of $T_1,\dots,T_{n-d}$ are chosen sufficiently close to each other, any bounded $d$-simplex of the arrangement that uses at least two of the hyperplanes $T_j$ has $d$-volume smaller than $d^d/d!$.
Hence the simplices bounded by $H_1,\dots,H_d$ and one $T_j$ are of maximum $d$-volume in the arrangement, and therefore $M_d(n)\ge n-d$.
\subsection{Explicit lower bound construction for $M_3 (n)$}

As mentioned in the introduction, Damasdi et al.~\cite{damasdi2020triangle} obtained the lower bound $M_2(n) > \frac{6}{5}n - O(1)$.
The idea of their lower bound construction is the following: They begin with a family of five lines forming a regular pentagon, and then recursively place arrangements together. At each step, two arrangements are combined by applying a suitable affine transformation so that the areas of their maximum area triangles coincide, and placing them in such a way that the maximum area does not increase.  Moreover, by sliding one of the whole families of lines, one can create two additional maximum area triangles. Consequently, we obtain a construction in which the number of maximum area triangles increases by seven whenever a family of five lines is added.
We see that this idea can also work in three dimensions.

\begin{prop} \label{prop:M_3}
$M_3 (n) > \frac{7}{6} n - O(1).$
\end{prop}
The following can be proved using the same procedure as in their proof. Details are left to the reader.



\begin{prop} \label{prop:4.2}
If there are no two parallel planes in the arrangement, then we can find a rectangular box $ABCDEFGH$ such that adding any plane to the arrangement that intersects the segments $AB$, $CD$, $EF$ and $GH$ creates no new maximum volume tetrahedra.
\end{prop}

For a plane arrangement $\mathcal{A}$, let $T(\mathcal{A})$ denote the number of maximum volume tetrahedra in $\mathcal{A}$.

\begin{figure}[htbp]
  \centering
  \begin{minipage}{0.48\textwidth}
    \centering
    \resizebox{\linewidth}{!}{%
\begin{tikzpicture}[x=0.75pt,y=0.75pt,yscale=-1,xscale=0.8]

\draw    (137.41,73.55) -- (322.5,75.54) ;
\draw    (271.44,131.76) -- (271.44,313) ;
\draw    (271.44,131.76) -- (332.53,86.78) ;
\draw    (345,76) -- (493,76.2) ;
\draw    (245.91,131.76) -- (271.44,131.76) ;
\draw    (350.21,52) -- (371,52) ;
\draw    (245.91,131.76) -- (245.91,313) ;
\draw    (245.91,313) -- (271.44,313) ;
\draw    (245.91,131.76) -- (306.09,87.44) ;
\draw    (371,52) -- (371,75) ;
\draw  [dash pattern={on 0.84pt off 2.51pt}]  (371,75) -- (371.74,234.95) ;
\draw    (332.53,86.78) -- (478.41,86.78) ;
\draw    (478.41,86.78) -- (493,76.2) ;
\draw    (332.53,86.78) -- (333.44,268.68) ;
\draw    (478.41,86.78) -- (478.41,268.02) ;
\draw    (493,76.2) -- (493,257.44) ;
\draw  [dash pattern={on 0.84pt off 2.51pt}]  (350.21,52) -- (351.68,235.61) ;
\draw    (478.41,268.02) -- (493,257.44) ;
\draw    (333.44,268.68) -- (478.41,268.02) ;
\draw  [dash pattern={on 0.84pt off 2.51pt}]  (137.41,256.11) -- (493,257.44) ;
\draw    (271.44,313) -- (333.44,268.68) ;
\draw    (121,85.46) -- (306.09,87.44) ;
\draw  [dash pattern={on 0.84pt off 2.51pt}]  (137.41,73.55) -- (137.41,256.11) ;
\draw  [dash pattern={on 0.84pt off 2.51pt}]  (122.82,266.7) -- (137.41,256.11) ;
\draw    (122.82,266.7) -- (245.91,268.02) ;
\draw    (345,76) -- (371,52) ;
\draw    (322.5,75.54) -- (350.21,52) ;
\draw    (121,85.46) -- (122.82,266.7) ;
\draw    (121,85.46) -- (137.41,73.55) ;
\draw  [dash pattern={on 0.84pt off 2.51pt}]  (245.91,268.02) -- (333.44,268.68) ;
\draw  [dash pattern={on 0.84pt off 2.51pt}]  (333.44,268.68) -- (371.74,234.95) ;
\draw  [dash pattern={on 0.84pt off 2.51pt}]  (245.91,313) -- (351.68,235.61) ;
\draw  [dash pattern={on 0.84pt off 2.51pt}]  (351.68,235.61) -- (371.74,234.95) ;

\draw (375,37.4) node [anchor=north west][inner sep=0.75pt]    {$A_{1}$};
\draw (318,36.4) node [anchor=north west][inner sep=0.75pt]    {$B_{1}$};
\draw (275,129.4) node [anchor=north west][inner sep=0.75pt]    {$D_{1}$};
\draw (219,127.4) node [anchor=north west][inner sep=0.75pt]    {$C_{1}$};
\draw (266,321.4) node [anchor=north west][inner sep=0.75pt]    {$H_{1}$};
\draw (220,319.4) node [anchor=north west][inner sep=0.75pt]    {$G_{1}$};
\draw (374,226.4) node [anchor=north west][inner sep=0.75pt]    {$E_{1}$};
\draw (308,221.4) node [anchor=north west][inner sep=0.75pt]    {$F_{1}$};
\draw (453,94.4) node [anchor=north west][inner sep=0.75pt]    {$A_{2}$};
\draw (495,55.4) node [anchor=north west][inner sep=0.75pt]    {$B_{2}$};
\draw (127,47.4) node [anchor=north west][inner sep=0.75pt]    {$C_{2}$};
\draw (91,81.4) node [anchor=north west][inner sep=0.75pt]    {$D_{2}$};
\draw (459,272.4) node [anchor=north west][inner sep=0.75pt]    {$E_{2}$};
\draw (496,240.4) node [anchor=north west][inner sep=0.75pt]    {$F_{2}$};
\draw (145,232.4) node [anchor=north west][inner sep=0.75pt]    {$G_{2}$};
\draw (95,267.4) node [anchor=north west][inner sep=0.75pt]    {$H_{2}$};

\end{tikzpicture}
    }
    \caption{Rectangular boxes}
    \label{Rectangular boxes}
  \end{minipage}\hfill
  \begin{minipage}{0.48\textwidth}
    \centering
    \resizebox{\linewidth}{!}{%
      \begin{tikzpicture}[x=0.75pt,y=0.75pt,yscale=-0.6, xscale=0.6]

\draw    (411,9) -- (329.22,129.49) ;
\draw    (411,9) -- (418.78,134.51) ;
\draw    (329.22,129.49) -- (357.77,173.42) ;
\draw    (357.77,173.42) -- (418.78,134.51) ;
\draw    (411,9) -- (357.77,173.42) ;
\draw    (202,119.45) -- (357.77,173.42) ;
\draw    (202,119.45) -- (329.22,129.49) ;
\draw    (202,119.45) -- (282.48,202.29) ;
\draw    (282.48,202.29) -- (357.77,173.42) ;
\draw    (282.48,202.29) -- (213.68,303.96) ;
\draw    (418.78,134.51) -- (546,144.56) ;
\draw    (357.77,173.42) -- (546,144.56) ;
\draw    (357.77,173.42) -- (430.47,314) ;
\draw    (357.77,173.42) -- (420.08,198.53) ;
\draw    (420.08,198.53) -- (430.47,314) ;
\draw    (546,144.56) -- (420.08,198.53) ;
\draw    (357.77,173.42) -- (346.09,241.2) ;
\draw    (357.77,173.42) -- (213.68,303.96) ;
\draw    (346.09,241.2) -- (220.17,302.7) ;
\draw    (430.47,314) -- (346.09,241.2) ;

      \end{tikzpicture}
    }
    \caption{A star-shaped badge}
    \label{A star-shaped badge}
  \end{minipage}
\end{figure}

\begin{prop}[$3$-dimensional analogue of Proposition 4.2 in \cite{damasdi2020triangle}]
If $\mathcal{H}_1$ and $\mathcal{H}_2$ are arrangements of planes neither of which contains two parallel planes, then there exist affine transformations $\varphi, \psi$ such that 
$$
T (\varphi (\mathcal{H}_1) \cup \psi (\mathcal{H}_2)) \geq T(\mathcal{H}_1) + T(\mathcal{H}_2) + 2.
$$
\end{prop}

\begin{proof}
In an arrangement $\mathcal{H}_1$, there exists a rectangular box satisfying Proposition~\ref{prop:4.2} with vertices $A_1, B_1, C_1, D_1, E_1, F_1, G_1, H_1$, and similarly in $\mathcal{H}_2$, there exists one with vertices $A_2, B_2, C_2, D_2, E_2, F_2, G_2, H_2$. By applying appropriate affine transformations to each, we can arrange $\mathcal{H}_1$ and $\mathcal{H}_2$ so that the volumes of their respective maximum volume tetrahedra are equal and the two rectangular boxes cross each other as shown in Figure \ref{Rectangular boxes}. Let $\varphi$ denote the affine transformation applied to $\mathcal{H}_1$ and $\psi$ the affine transformation applied to $\mathcal{H}_2$. 

As the next step, we translate the planes of $\psi(\mathcal{H}_2)$ to create a new maximum-volume tetrahedron $S$. If exactly one of the four planes defining $S$ comes from a plane of $\varphi(\mathcal{H}_1)$, let this plane be $H$. By moving the entire $\psi(\mathcal{H}_2)$ along $H$, we can create another new maximum-volume tetrahedron. Similarly, if exactly three planes come from planes of $\varphi(\mathcal{H}_1)$, we can move $\varphi(\mathcal{H}_1)$ along one plane of $\psi(\mathcal{H}_2)$. Finally, if exactly two planes come from planes of $\varphi(\mathcal{H}_1)$, by moving $\psi(\mathcal{H}_2)$ along the line of intersection of those two planes, we can create a new maximum-volume tetrahedron.
\end{proof}

We observe that the arrangement of planes $\mathcal{S}\subset \mathbb{R}^3$ forming a star-shaped badge as shown in Figure~\ref{A star-shaped badge}
consists of six planes and forming five maximum volume tetrahedra.
Starting from $\mathcal{S}$, we repeatedly add suitably transformed copies of $\mathcal{S}$ so that the above proposition applies at each step.
Consequently, for every $k\ge 1$ we obtain an arrangement of $6k$ planes in $\mathbb{R}^3$ that determines at least $7k-2$ maximum volume tetrahedra,
which proves Proposition~\ref{prop:M_3}.

\subsection{General lower bound for $M_d (n)$}
We improve here the lower bound achieved by Proposition~\ref{prop:M_3} and also provides lower bounds for all $d \geq 3$.

\Md*

The Theorem follows immediately from the following observation.
\begin{obs}
For all $d \geq 3$, we have 
$$
M_d (n) \geq M_{d-1} (n-1).
$$
\end{obs}

\begin{proof}
Fix a hyperplane $H\subset \mathbb{R}^d$.  Choose $(d-2)$-flats (affine subspaces of dimension $d-2$)
$
A_1,\ldots,A_{n-1}
$
inside $H \ (\cong \mathbb{R}^{d-1})$ so that the number of $(d-1)$-simplices of maximum $(d-1)$-volume
formed by $A_1,\ldots,A_{n-1}$ equals $M_{d-1}(n-1)$.
Fix a point $p\in \mathbb{R}^d\setminus H$, and for each $i\in\{1,\ldots,n-1\}$
let $H_i$ be the hyperplane containing $p$ and $A_i$.
Set $H_n:=H$.  This yields an arrangement of $n$ hyperplanes in $\mathbb{R}^d$,
$
\mathcal{H}:=\{H_1,\ldots,H_n\}
$.
Then every $d$-simplex $\Delta$ in $\mathcal{H}$ have exactly one facet contained
in $H$ (i.e., $\Delta\cap H$ is a $(d-1)$-simplex).  Conversely, every $(d-1)$-simplex $F$ in the arrangement $\{A_1,\ldots,A_{n-1}\}\subset H$ arises as the
facet $F=\Delta\cap H$ of a unique bounded $d$-simplex $\Delta$ in $\mathcal{H}$.
Moreover, whenever $\Delta$ corresponds to $F=\Delta\cap H$, the simplex $\Delta$ is a
cone with base $F$ and apex $p$.
Hence, its $d$-volume is a product of the $(d-1)$-volume of $F$ and $\operatorname{dist}(p,H)/d$.
\end{proof}

We are not aware of any upper bound for $M_d(n)$ that holds for all $d \geq 3$. Moreover, we do not have any reason to believe that the correct upper bound is also close to linear. Whether $M_d (n)$ is linear in $n$ is open even in the case $d=2$. Note that Dam\'{a}sdi et al.~\cite{damasdi2020triangle} conjectured that $M_2(n) = O(n^{1+\varepsilon})$ for every $\varepsilon>0$.

\section{Hyperplanes defining distinct volume $d$-simplices} \label{section5}

In this section, we show that for all $d \geq 2$, $D_d (n)$ grows strictly slower than $n$.
\Dd*

Let $[n] = \{1,\ldots, n \}$ and denote by $r_k (n)$ the size of the largest $S \subseteq [n]$ such that $S$ contains no $k$ elements in arithmetic progression. The following is the key proposition.

\begin{prop} \label{Prop13}
$D_d (n) < r_{d+2} (n)$ for all $d \geq 2$.
\end{prop}

Determining the value of $r_k(n)$ is one of the fundamental problems in arithmetic combinatorics, and both upper and lower bounds have been progressively improved. 
In 2009, Green and Tao~\cite{green2006new, green2017new} proved that there exists a constant $c > 0$ such that
$r_4(n) \ll n (\log n)^{-c}$.
For $k \geq 4$, Leng et al.~\cite{leng2024improved} recently obtained an upper bound $r_k(n) << n e^{-(\log \log n)^{c_k}}$ for some $c_k \in (0,1)$.
See \cite{szemeredi1975sets, gowers1998new, gowers2000arithmetic} for further details.
Theorem \ref{Dd} can be immediately derived from these bounds and Proposition \ref{Prop13}. 
Therefore, any improvement of upper bound for $r_k (n)$ automatically improves that for $D_d (n)$.

\begin{proof}[Proof of Proposition \ref{Prop13}]
We first consider the case $d=2$.
Define $\mathcal{L} = \{ l_1, \dots, l_n \}$ as the set of lines corresponding to $n$ consecutive edges of a regular $N$-gon $(N > n)$, arranged in clockwise order.
We choose $N$ sufficiently large so that any four lines in $\mathcal{L}$ form two triangles.
By the definition of $r_k(n)$, every subset of $\mathcal{L}$ of size $r_4(n) + 1$ always contains four lines $\{ l_{i_1}, l_{i_2}, l_{i_3}, l_{i_4} \}$ such that $\{ i_1, i_2, i_3, i_4 \} \subseteq [n]$ is a $4$-terms arithmetic progression.
One can easily see that these four lines then form two congruent triangles.

We generalize this construction to higher dimensions. First, assume that $d$ is even.
Write $\mathbb{R}^d$ as the orthogonal sum of two-dimensional subspaces:
$$
\mathbb{R}^d = V_1 \oplus \cdots \oplus V_{d/2}.
$$
For each $j$, let $C_j$ denote the unit circle centered at the origin in $V_j$.
We choose angle parameters $\theta_1, \ldots, \theta_{d/2} > 0$ so that for each $j$, one has
$
n\theta_j \ll \pi/2,
\, \text{and} \,\, \theta_j \neq \theta_{j'}
\ \text{for } j \neq j'.
$
For each $i \in \{0, \ldots, n\}$ and each $j \in \{1, \ldots, d/2\}$, we define
$
P_{i,j} = (\cos(i\theta_j), \sin(i\theta_j)) \in V_j
$
(viewed in the standard coordinates of $\mathbb{R}^d$, the point $P_{i,j}$ has nonzero coordinates only in its $V_j$-component, which is $(\cos(i\theta_j), \sin(i\theta_j))$, while all other components are zero).
For each $i \in \{1, \ldots, n\}$, let $H_i$ denote the affine hull of
$$
\bigcup_{1 \le j \le d/2} \{ P_{i-1,j}, P_{i,j} \}.
$$
This has dimension $d-1$.
Let $\mathcal{H} = \{ H_1, \ldots, H_n \}$ denote the resulting hyperplane arrangement.
One can verify that $\mathcal{H}$ is in general position for almost all choices of the angle parameters $\theta_1, \ldots, \theta_{d/2}$.
For any positive integer $D$, we define the map
$T_D : \mathbb{R}^d \to \mathbb{R}^d$ as 
$$
T_D|_{V_j} = R_{D\theta_j} \, (j = 1, \ldots, d/2),
$$
where $R_{D\theta_j}$ is the rotation by angle $D\theta_j$ on $V_j$.
One easily sees that 
$
T_D(P_{i,j}) = P_{i+D,j}
$ (for $i+D \leq n$),
and therefore
$
T_D(H_i) = H_{i+D}.
$
Since the Jacobian matrix of $T_D$ is 
$$
J_{T_D} = 
\begin{pmatrix}
R_{D\theta_1} & 0 & \cdots & 0\\
0 & R_{D\theta_2} & \ddots & \vdots\\[2pt]
\vdots & \ddots & \ddots & 0\\[2pt]
0 & \cdots & 0 & R_{D\theta_{d/2}}
\end{pmatrix},
$$

$
\det J_{T_D}=\prod_{j=1}^{d/2} \det R_{D\theta_j} = 1.
$
Thus, the map $T_d$ is an affine volume-preserving map.

Now let
$
i_1 < i_2 < \cdots < i_{d+2}
$
be a $(d+2)$-arithmetic progression in $[n]$ with common difference $L$.
As the arrangement $\mathcal{H}$ is in general position, the hyperplanes 
$\{ H_{i_1}, \ldots, H_{i_{d+1}} \}$ 
form a unique $d$-simplex $S$, and the hyperplanes
$\{ H_{i_2}, \ldots, H_{i_{d+2}} \}$ 
form a unique $d$-simplex $S'$.
Since $T_L (H_{i_k}) = H_{i_{k+1}}$, we have
$
T_L (S) = S'
$
which means that $S$ and $S'$ have the same $d$-volume.
Therefore, by the definition of $r_k (n)$, every subfamily of $\mathcal{H}$ of size (at least) $r_{d+2}(n)+1$ always contain two $d$-simplices with the same $d$-volume.

Now, we assume $d$ is odd, we decompose $\mathbb{R}^d$ into an orthogonal direct sum of $(d-1)/2$ 2-dimensional vector spaces and one 1-dimensional vector space:
$$
\mathbb{R}^d = V_1 \oplus \cdots \oplus V_{(d-1)/2} \oplus L.
$$
First, we place $n$ equally spaced points $Q_i = (0, \ldots, 0, i)$ ($i \in \{1, \ldots, n\}$) on $L$.
Choose angular parameters $\theta_1, \ldots, \theta_{(d-1)/2} > 0$ such that $\theta_j \neq \theta_{j'}$ for $j \neq j'$.
For each $j \in \{1, \ldots, (d-1)/2\}$, consider the helical curve wrapping around $L$ in the $3$-dimensional vector space $V_j \oplus L$:
$$
\gamma_{\theta_j}(t) = (0,\ldots, 0,\underbrace{\cos (t \theta_j)}_{2j-1}, \underbrace{\sin (t \theta_j)}_{2j}, 0, \ldots, 0, t)
$$
and we place $(n+1)$ equally spaced points on $\gamma_{\theta_j}(t)$:
$$
P_{i, j} = (0,\ldots, 0,\underbrace{\cos (i \theta_j)}_{2j-1}, \underbrace{\sin (i \theta_j)}_{2j},0,\ldots 0, i) ,\quad i \in \{0, \ldots, n\}.
$$
Define $H_i$ as the affine hull of
$$
\bigcup_{1 \leq j \leq (d-1)/2} \{ P_{i-1, j}, P_{i,j} \} \cup \{Q_i\},
$$
which has dimension $d-1$.
By taking $\theta_1, \ldots, \theta_{(d-1)/2}$ appropriately,
we obtain the hyperplane arrangement $\mathcal{H} = \{H_1, \ldots, H_n\}$ which is in general position.
For any positive integer $D$, one can define an affine volume-preserving map $T_D :\mathbb{R}^d \rightarrow \mathbb{R}^d$ by 
$$
T_D|_{V_j} = R_{D\theta_j}, \, T_D|_L : m \mapsto m+D.
$$
Observe that
$
T_D(P_{i,j}) = P_{i+D,j},
\,
T_D((0, \ldots, 0, i)) = (0, \ldots, 0, i+D)
$ (for $i+D \leq n$),
hence
$
T_D(H_i) = H_{i+D}.
$
The same argument as in the previous case shows that any subfamily of $\mathcal{H}$ with size $r_{d+2} (n) + 1$ contains two $d$-simplices with the same $d$-volume.

\end{proof}

\begin{rem}
As previouly mentioned, Dam\'{a}sdi et al. \cite{damasdi2020triangle} obtained the lower bound with some additional assumptions: if there are no six lines that are tangent to a common conic  (observe that five lines always have a 
common tangent conic), then an arrangement of $n$ lines always contains a subset of size $\Omega (n^{1/5})$ whose induced triangles have all distinct $d$-volumes.
Their argument can be applied to general dimensions, and the following can be shown:
If there are no $\binom{2d}{d}$ hyperplanes that are tangent to a common hypersurface, then 
then an arrangement of $n$ hyperplanes always contains a subset of size $\Omega_d (n^{\frac{1}{2d+1}})$ whose induced $d$-simplices have all distinct volumes.
However, We are not aware of any lower bounds for $D_d(n)$ for $d \geq 2$.
\end{rem}

\medskip
\noindent {\bf Acknowledgement.} I would like to thank Pavel Valtr for helpful suggestions.
The work was supported by grant no. 23-04949X of the Czech Science Foundation (GAČR) and the Charles University Grant Agency (GAUK) project number 378426.

\typeout{}
\bibliography{Simplex_volumes_in_hyperplane_arrangements_revised }
\bibliographystyle{plainurl}

\appendix 
\section*{Appendix}
\begin{proof}[Proof of Proposition \ref{Kobon}]
It suffices to prove that
$$
K_d(n)\le \frac{2n}{d+1}K_{d-1}(n-1)
$$
holds for every $d \geq 3$.
Let $\mathcal{H}$ be an arrangement of $n$ hyperplanes in $\mathbb{R}^d$. Fix a hyperplane $H \in \mathcal{H}$. The restriction of $\mathcal{H}$ to $H$ induces an arrangement in $H \simeq \mathbb{R}^{d-1}$, namely
$$
\mathcal H_H:=\{H\cap H' : H'\in \mathcal H,\ H'\neq H \}.
$$

Let $F \subset H$ be a facet of a simplicial $d$-cell of $\mathcal{H}$. Since the interior of $F$ in $H$ is not intersected by any hyperplane of $\mathcal H\setminus\{H\}$,  $F$ is the closure of a simplicial $(d-1)$-cell of $\mathcal{H}_H$.
Each simplicial $(d-1)$-cell of $\mathcal{H}_H$ is incident to at most two simplicial $d$-cells of $\mathcal{H}$. 
Hence, the number of simplicial $d$-cells of $\mathcal{H}$ having a facet contained in $H$ is at most
$$
2K_{d-1}(n-1).
$$

we now take the sum over all $H \in \mathcal{H}$. Since each $d$-simplex in $\mathcal{H}$ has exactly $d+1$ faces, and each of these faces lies in one hyperplane of $\mathcal{H}$, we have,
$$
(d+1)K_d(n)\le 2nK_{d-1}(n-1).
$$
\end{proof}
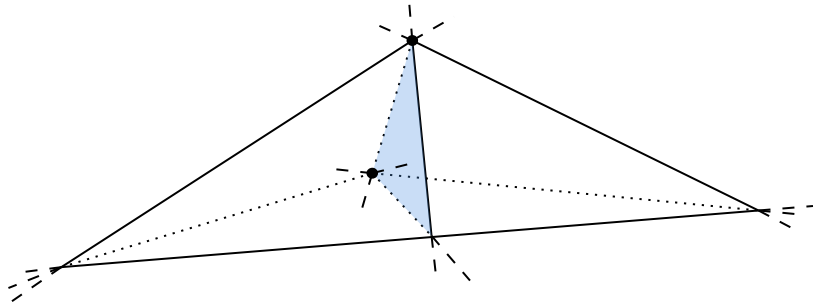
\end{document}